\documentclass{amsproc}

\usepackage{amsmath,amsfonts,amssymb,amsthm}
\usepackage{epsfig}
\usepackage[all]{xy}
\usepackage{color}
\newtheorem{theorem}{Theorem}[section]
\newtheorem{lemm}[theorem]{Lemma}
\newtheorem{prop}[theorem]{Proposition}
\newtheorem{theo}[theorem]{Theorem}
\theoremstyle{definition}
\newtheorem{defi}[theorem]{Definition}

\newtheorem{coro}[theorem]{Corollary}
\theoremstyle{remark}
\newtheorem{remark}[theorem]{Remark}

\numberwithin{equation}{section}

\newcommand{\la}{\lambda}

\newcommand{\s}{\sigma}

\def\ra{\rangle}
\def\la{\langle}

\begin{document}

\title[Twisted Drinfeld realizations]
    {Drinfeld realization of quantum twisted affine algebras via braid group}

\author[Jing]{Naihuan Jing}
\address{
   Department of Mathematics, North Carolina State University,
Raleigh, NC 27695, USA} \email{jing@math.ncsu.edu}

\author[Zhang]{Honglian Zhang$^{\star}$}
\address{Department of Mathematics,
Shanghai University, Shanghai 200436, China}
\email{hlzhangmath@shu.edu.cn}
\thanks{$^\star$H.Z., Corresponding Author}


\begin{abstract}
The Drinfled realization of quantum affine algebras has been tremendously useful since its discovery.
Combining techniques of Beck and Nakajima with our previous approach,
we give a complete and conceptual proof of the Drinfeld realization for the twisted quantum affine algebras
using Lusztig's braid group action.
\end{abstract}

\maketitle

\section{ Introduction}

In studying finite dimensional representations of Yangian algebras and quantum affine algebras, Drinfeld gave a new realization of the Drinfeld-Jimbo quantum enveloping algebras of the affine types. Drinfeld realization is a quantum analog
of the loop algebra realization of the affine Kac-Moody Lie algebras, and has played a pivotal role in later developments of quantum affine algebras and the quantum conformal field theory. For example the basic representations of quantum affine algebras
were constructed based on Drinfeld realization \cite{FJ, J1} and the quantum Knizhnik-Zamoldchikov equation
\cite{FR} was also formulated using this realization.

The first proof of the Drinfeld realization for the untwisted types was given by Beck \cite{B} using Lusztig's braid group actions (see also \cite{LSS} for $U_q(\widehat{sl}_2)$). For
other approaches see \cite{KT, ER, DF}. By directly quantizing the
classical isomorphism of the Kac realization to the affine Lie algebras, the first author later gave an elementary
proof \cite{J2} of the Drinfeld automorphism using $q$-brackets starting from Drinfeld's quantum loop algebras.
In this elementary
approach a general strategy and algorithm was
formulated to prove the Serre relations. In particular,
all type $A$ relations were verified in details including the exceptional type $D^{(3)}_4$.

 In \cite{ZJ, JZ},
we gave an elementary proof of the twisted Drinfeld realization starting from Drinfeld's quantum loop algebras.
In particular in
\cite{JZ} we have shown that twisted quantum affine algebras obey some simplified Serre relations
in certain types as in the classical cases, and from which twisted Serre relations are consequences.

The purpose of this work is to give a conceptual proof of twisted Drinfeld realizations
using braid group actions
starting from the Drinfeld-Jimbo definition. Braid group actions have been very useful in Lusztig's construction of
canonical bases \cite{L1, BN} and are particularly useful in Beck's proof
of untwisted cases. We use the extended braid group action to define root vectors in the twisted
case as in the untwisted cases. Some of the root vector computations can be
 done in a similar way as in untwisted cases. We also give a complete proof of all Serre relations
 for both untwisted and twisted cases using $q$-bracket
 techniques \cite{J1}. One notable feature of our work is that we directly prove the
 realization using the braid group action and verify the unchecked relations for
 all cases.

The second goal
of our paper
is to provide a different proof that the Drinfeld-Jimbo algebra is indeed isomorphic
to the Drinfeld algebra of the quantum affine algebra without passing to
$q=1$ case. We achieve this by combining our previous
work identifying of the Drinfeld-Jimbo quantum affine algebra inside the Drinfeld realization 
and the explicit knowledge of the $q$-bracket computations developed in \cite{J2, ZJ}, which also established the epimorphism from the Drinfeld-Jimbo form to the Drinfeld form of the quantum affine algebra.

The paper is organized in the following manner. In Section two, we first recall the Drinfeld-Jimbo quantum affine enveloping algebras and define Lusztig's braid group action as well as the extended braid group action, and then
use this to define quantum root vectors.
Section three shows how to construct the quantum affine algebra $U_q(A_1^{(1)})$ (or $U_q(A_2^{(2)})$)
inside the twisted quantum affine algebra
$U_q(X^{(r)})$ (or $U_q(A_{2n}^{(2)})$ ). In Section four, we check all Drinfeld relations, in particular all Serre relations for twisted quantum affine algebras. Finally we prove the isomorphism of two forms of quantum affine algebras using our previous work on Drinfeld realization and $q$-bracket techniques.

\section{Definitions and Preliminaries}

\subsection{Finite order automorphisms of $\frak{g}$}
 In this paragraph some basic notations of Kac-Moody Lie algebras are recalled \cite{K}. Let $\frak{g}$ be a simple
 finite-dimensional Lie algebra, and let $\sigma$ be an (outer) automorphism of $\frak{g}$ of order $r$. Then $\sigma$ induces an automorphism of the Dynkin diagram of $\frak{g}$ with the same order. Fix a primitive $r$th root of unity $\omega=exp\frac{2\pi i}{r}$. Since $\sigma$ is diagonalizable, it follows that $$\frak{g}=\bigoplus_{j\in \mathbb{Z}/r\mathbb{Z}}\frak{g}_{_j},$$
 where $\frak{g}_{_j}$ is the $\sigma$-eigenspace with eigenvalue $\omega^j$. Clearly, the decomposition is a $\mathbb{Z}/r\mathbb{Z}-$ gradation of $\frak{g}$, then $\frak{g}_{_0}$ is a Lie subalgebra of $\frak{g}$.

Let $A=(A_{ij}), (i,\,j \in \{1,\,2,\,\cdots,\, N\})$ be one of the simply laced Cartan
matrices.
It is well-known that the Dynkin diagram $D(A)$ has a diagram automorphism $\sigma$ of order
$r=2$ or $3$. Explicitly $A$ is one of the following types: $A_N\, (N\geqslant 2)$, $D_N\, (N>4)$, $E_6$ and
$D_4$ with the canonical action of $\sigma$:
\begin{equation*}
\begin{split}
A_N: &\, \sigma(i)=N+1-i,\\
D_N: &\, \sigma(i)=i, 1\leqslant i\leqslant N-2; \, \sigma(N-1)=N,\\
E_6: &\, \sigma(i)=6-i, 1\leqslant i\leqslant 5; \, \sigma(6)=6,\\
D_4: &\, \sigma(1,2,3,4)=(3,2,4,1).
\end{split}
\end{equation*}

Let $I=\{1,\,2,\,\cdots,\,n\}$ be the set of $\sigma-$orbits on $\{1,\,2,\,\cdots,\, N\}$, where we
use representatives to denote the orbits. For example $\{i, N+1-i\}$ is simply denoted by $i$ in type $A$.
We can write $\{1, \dots, N\}=I\cup \sigma(I)$. Consequently the nodes of the Dynkin diagram $\frak{g}_{_0}$ are indexed by $I$.

\subsection{Twisted affine Lie algebras}\,
 For a non trivial automorphism $\sigma$ of the Dynkin diagram, the twisted affine Lie algebra $\widehat{\frak{g}}^{\sigma}$ is the central extension of the twisted loop algebra:
$$\widehat{\frak{g}}^{\sigma}=\Big(\bigoplus_{j\in \mathbb{Z}}\frak{g}_{_{[j]}} \otimes \mathbb{C}t^j\Big)\oplus \mathbb{C}c
\oplus \mathbb{C}d,$$
where $c$ is the central element and $ad(d)=t\frac{d}{dt}$.
Denote by $\hat{I}=I \bigcup \{0\}$ the node set of the Dynkin diagram
of $\widehat{\frak{g}}^{\sigma}$. Let $\widehat{\mathfrak
h}^{\sigma}=\mathfrak h_{[0]}\oplus \mathbb Cc\oplus \mathbb Cd$ be
the Heisenberg subalgebra of $\widehat{\frak{g}}^{\sigma}$.

Let us use $A^{\sigma}=(a_{ij})\, (i,\,j\in \hat{I})$ to denote the
Cartan matrix of the twisted affine Lie algebra
$\widehat{\frak{g}}^{\sigma}$ of type $X_N^{(r)}$. The Cartan matrix
$A^{\sigma}$ is symmetrizable, that is, there exists a diagonal
matrix $D=diag(d_i|i\in \hat{I})$ such that $DA^{\sigma}$ is
symmetric. Let $\alpha_i\, (i\in \hat{I})\subset \widehat{\mathfrak
h}^{\sigma *}$ be the simple roots and let $\alpha_i^{\lor}\, (i\in
\hat{I})\subset \widehat{\mathfrak h}^{\sigma}$ be the simple
coroots of $\widehat{\frak{g}}^{\sigma}$ such that $\langle
\alpha_j, \alpha_i^{\lor}\rangle=a_{ij}$. Let $Q$ be the affine root
lattice defined by:
$$Q= \bigoplus_{i\in \hat{I}}\mathbb{Z}\alpha_i.$$
Subsequently $Q_0$ will denote the finite root lattice of the
subalgebra $\frak{g}_{_0}$. Indeed, $Q_0\subseteq Q=Q_0\oplus
\mathbb{Z}\alpha_0$. Let $Q^+=\bigoplus_{i\in
\hat{I}}\mathbb{Z}_{\geq 0}\alpha_i$, then $Q^+_0=Q_0\cap Q^+$.

Introduce the non-degenerate symmetric bilinear form $(\, , \, )$ on $Q$ determined by $(\alpha_i,\,\alpha_j)=d_i a_{ij}$. Let $\delta=\sum_{i\in \hat{I}}r_i\alpha_i\in Q^+$ be the canonical imaginary root of minimal height
such that $(\delta,\,\delta)=0$ and $(\delta,\, \alpha_i)=0, \forall i\in \hat{I}$.  Here the coefficients $r_i$ are unique such that $r_0$ is always $1$, and we have chosen the labels of $A_{2n}^{(2)}$ different
from \cite{K}.

Let $P^{{\lor}}=\{\lambda^{\lor}\in \widehat{\mathfrak h}^{\sigma}| \langle \alpha_i, \lambda^{\lor}\rangle\in \mathbb Z\}$ be the coweight lattice over $\mathbb{Z}$, and define the fundamental coweights
$\omega_i^{\lor}\in P^{{\lor}}$
such that $\langle \alpha_i,\, \omega_j^{\lor}\rangle=\delta_{ij}, \langle c, \omega_j^{\lor}\rangle =0, \forall i\in \hat{I}$. So $d=\omega_0^{\lor}$. Additionally, we are going to introduce another important sublattice $P$ defined by $P=\oplus_{i\in \hat{I}}\mathbb{Z}\omega_i$, where $\omega_i=p_i\omega_i^{\lor}$ for $i\in I$ and $\omega_0=\omega_i^{\lor}=d$, and
$$p_i=\begin{cases}
r, \quad \hbox{if} \quad \sigma(i)=i;\vspace{3pt}\\
1, \qquad \hbox{otherwise}
\end{cases}$$
Let us denote by $P_0^{{\lor}}$ and $P_0$ the finite coweight  and
weight lattice of the subalgebra $\frak{g}_{_0}$ respectively. Then
$P_0^{{\lor}}\subseteq P^{{\lor}}=P_0^{{\lor}}\oplus
\mathbb{Z}\omega_0^{{\lor}}$ and $P_0^{{\lor}}\subseteq
P^{{\lor}}=P_0^{{\lor}}\oplus \mathbb{Z}\omega_0$.

 Recall that the (affine) Weyl group $W$ is generated by $\{s_i|i\in \hat{I}\}\leq Aut(\widehat{\mathfrak h}^{\sigma*})$, where $s_i(\lambda)=
 \lambda-\langle \lambda, \alpha_i^{\lor}\rangle\alpha_i$, thus $s_i(\alpha_j)=\alpha_j-a_{ij}\alpha_i$. The
 Weyl group $W$ acts on $P^{\lor}$ by
 $s_i(\lambda^{\lor})=
 \lambda^{\lor}-\langle \alpha_i, \lambda^{\lor}\rangle \alpha_i^{\lor}$. For any $w\in W$, if $w=s_{i_1}s_{i_2}\ldots s_{i_l}$
 is a reduced expression of $w$, then we define the length $l(w)=l$. Let $W_0$ be the subgroup
 generated by $\{s_i|i\in {I}\}$, thus $W_0$ is the finite Weyl group of $\frak{g}_{_0}$.

 Let $Aut(\Gamma)$ be the group of automorphisms of the Dynkin diagram associated with  $\widehat{\frak{g}}^{\sigma}$. Then $Aut(\Gamma)$ acts on $W$ by $\tau s_i\tau^{-1}=s_{\tau(i)},\, \forall i\in \hat{I},\, \forall \tau\in Aut(\Gamma)$. We denote $\mathcal{T}=Aut(\Gamma)\cap (W_0 \ltimes P^{\lor})$. We further introduce the extended affine Weyl group $\tilde{W}=\mathcal{T}\ltimes W$, which can be written as $W_0 \ltimes P$ for twisted cases. The length function $\ell$ is extended to $\tilde{W}$  by $\ell(\tau w)=\ell(w)$, for $\tau\in \mathcal{T},\, w\in {W}$.

\subsection{Twisted quantum affine algebras}
In this paragraph let us review the definition of the twisted quantum
affine algebra $U_q(\widehat{\frak{g}}^{\sigma})$. Set
$q_i=q^{d_i},\, i\in \hat{I}$. Introduce the $q-$integer by:
$$[n]_i=\frac{q_i^n-q_i^{-n}}{q_i-q_i^{-1}},\, [n]_i!=\prod_{k=1}^{n}[k_i].$$

\begin{defi} The twisted quantum affine algebra $U_q(\widehat{\frak{g}}^{\sigma})$ is an associative $\mathbb{C}(q)-$algebra generated by $\{E_i,\,F_i,\, K_i^{\pm1},\, \gamma^{\pm \frac{1}{2}}\}$ for $i \in \hat{I}$, satisfying the following relations:
\begin{eqnarray*}
&(R 1)\qquad\quad& \gamma^{\pm \frac{1}{2}}\, \hbox{is \, central and} \, \gamma=K_{\delta}  \\
&(R 2)\qquad\quad& K_i K_i^{-1}=1,\quad K_i K_j=K_j K_i ,\, \forall\, i,\,j\in \hat{I},  \\
&(R 3)\qquad\quad& K_i E_j K_i^{-1}=q_i^{a_{ij}}E_j, \quad\, K_i F_j K_i^{-1}=q_i^{-a_{ij}}F_j,\, \forall\, i,\,j\in \hat{I},  \\
&(R 4)\qquad\quad& [\,E_i,\,F_j\,]=\delta_{ij}\frac{K_{i}-K_{i}^{-1}}
{q_{i}-q_{i}^{-1}},  \, \forall\,  i, \, j \in \hat{I},  \\
&(R 5)\qquad\quad& \sum_{s=0}^{1-a_{ij}}(-1)^{s}\Big[{1-a_{ij}\atop
s}\Big]_{i}
E_{i}^{1-a_{ij}-s}E_j E_{i}^{s}=0, \, \forall\, i \neq j, \hspace{3cm}\\
&(R 6)\qquad\quad&  \sum_{s=0}^{1-a_{ij}}(-1)^{s}\Big[{1-a_{ij}\atop
s}\Big]_{i} F_{i}^{1-a_{ij}-s}F_j F_{i}^{s}=0, \, \forall\, i
\neq j .
\end{eqnarray*}

\end{defi}
\begin{remark}
\,(1)\,There exists a unique { Hopf algebra} structure on
$U_q(\widehat{\frak{g}}^{\sigma})$ with the comultiplication $\Delta$,
~counit $\varepsilon$ and antipode $S$ defined by ($i\in \hat{I}$):
\begin{eqnarray*}
&\Delta(E_{i})=E_{i}\otimes K_i^{-1}  +1 \otimes E_{i} ,  \quad
\Delta(F_{i})=F_{i}\otimes 1 +  K_i\otimes F_{i}
,  \quad \Delta(K_i)=K_i\otimes K_i , \\
&\varepsilon(E_{i})=0 ,  \quad \varepsilon(F_{i})=0 ,  \quad
\varepsilon(K_i)=1, \\
&S(E_{i})=-E_i K_{i},  \quad S(F_{i})=-K_{i}^{-1}F_i ,
 \quad S(K_i)=K_i^{-1}.
\end{eqnarray*}
\,(2)\,Let $U_q^{+}$ (respectively, $ U_q^{-}$) be the subalgebra
of $U_q(\widehat{\frak{g}}^{\sigma})$ generated by the elements
$E_i$ (respectively, $ F_i$) for $i\in \hat{I}$, and let $U_q^{0}$
be the subalgebra of $U_q(\widehat{\frak{g}}^{\sigma})$ generated by
$K_i$ and $\gamma^{\pm \frac{1}{2}}$. The twisted quantum affine
algebra  $U_q(\widehat{\frak{g}}^{\sigma})$ has the {triangular
decomposition:} $$U_q(\widehat{\frak{g}}^{\sigma})\cong U_q^{-} \otimes
U_q^{0} \otimes U_q^{+}.$$
\,(3)\, When $a_{ij}=-2$, the quantum
Serre relation becomes:
 \begin{equation*}
\begin{split}
&E_jE_i^3-[3]_iE_iE_jE_i^2+[3]_iE_i^2E_jE_i-E_i^3E_j\\
&=[\,E_jE_i^2-[2]_{2i}E_iE_jE_i+E_i^2E_j,\, E_i\,]
\end{split}
\end{equation*}
Similarly, if $a_{ij}=-3$, we have the following equality.
 \begin{equation*}
\begin{split}
&E_jE_i^4-[4]_iE_iE_jE_i^3+\frac{[4]_i[3]_i}{[2]_i}E_i^2E_jE_i^2-[4]_iE_i^3E_jE_i+E_i^4E_j\\
&=AE_i^2-[2]_iE_iAE_i+E_i^2A
\end{split}
\end{equation*}
where $A=E_jE_i^2-[2]_{i}E_iE_jE_i+E_i^2E_j$.

From the above identities, it is not hard to see that the higher
degree Serre relations can be derived from the lower degree ones. Also it
is expected that the non-simply-laced Dynkin
diagrams can be obtained from simply-laced ones by the action of $\sigma$.
\end{remark}

The following proposition can be checked easily.
\begin{prop}
There exist an antiautomorphism $\Phi$ (over $\mathbb C(q)$) and  an automorphism $\Omega$ (over $\mathbb C$) of $U_q(\widehat{\frak{g}}^{\sigma})$ defined as follows:
 \begin{equation*}
\begin{split}
\Phi(E_i)=F_i,\quad \Phi(F_i)=E_i,\quad \Phi(K_i)=K_i^{-1},\quad \Phi(\gamma^{\pm \frac{1}{2}})=\gamma^{\mp \frac{1}{2}}, \Phi(q)=q^{-1}\\
\Omega(E_i)=E_i,\quad \Omega(F_i)=F_i,\quad \Omega(K_i)=K_i^{-1},\quad \Omega(\gamma^{\pm \frac{1}{2}})=\gamma^{\mp \frac{1}{2}}, \Omega(q)=q^{-1}.
\end{split}
\end{equation*}
\end{prop}
We recall Lusztig's braid group associated to $W$, which acts as an automorphism group of the twisted quantum affine algebra $U_q(\widehat{\frak{g}}^{\sigma})$ (\cite{L1}). For simplicity, we denote by $T_i=T_{s_i},\, i\in \hat{I}$. The actions of $T_i$ and $T_i^{-1}$ on Chevalley generators  are defined as follows:
\begin{eqnarray*}
&&T_i(E_{i})=-F_i K_i, \quad T_i(F_{i})=-K_i^{-1}E_i,\\
&&T_i(E_j)=\sum_{s=0}^{-a_{ij}}(-1)^{s-a_{ij}}q_i^{-s}E_i^{(-a_{ij}-s)}E_j E_i^{(s)},\, \forall \, i\neq j\in \hat{I},\\
&&T_i(F_j)=\sum_{s=0}^{-a_{ij}}(-1)^{s-a_{ij}}q_i^{s}F_i^{(s)}F_j F_i^{(-a_{ij}-s)}, \,\forall \, i\neq j\in \hat{I}\\
&&T_i^{-1}(E_{i})=- K_i^{-1} F_i, \quad T_i^{-1}(F_{i})=-E_i K_i,\\
&&T_i^{-1}(E_j)=\sum_{s=0}^{-a_{ij}}(-1)^{s-a_{ij}}q_i^{-s}E_i^{(s)}E_j E_i^{(-a_{ij}-s)},\, \forall \, i\neq j\in \hat{I},\\
&&T_i^{-1}(F_j)=\sum_{s=0}^{-a_{ij}}(-1)^{s-a_{ij}}q_i^{s}F_i^{(-a_{ij}-s)}F_j F_i^{(s)}, \,\forall \, i\neq j\in \hat{I}
\end{eqnarray*}
where $E_i^{(s)}=\frac{E_i^s}{[s]_i!}$ and $F_i^{(s)}=\frac{F_i^s}{[s]_i!}$.

We extent the braid group action to the extended Weyl group $\tilde{W}$ by defining $T_{\tau}$ as:
$$T_{\tau}(E_i)=E_{\tau(i)},\,T_{\tau}(F_i)=F_{\tau(i)},\, T_{\tau}(K_i)=K_{\tau(i)}.$$

\begin{remark}\, The braid group action $T_i$ commutes with $\Phi$, that is $T_i\Phi=\Phi T_i$.
Furthermore, one has $\Omega T_i=T_i^{-1}\Omega$.
\end{remark}

Note that $\tilde{W}=W_0 \ltimes P$, some properties about
$T_{\omega_i}$ will be discussed which are almost the same to non
twisted cases, where $\omega_i=p_i \omega_i^{\lor}$.

Finally for each $i\in\hat{I}$, we define inductively
the twisted derivation $r_i$ of $U_q^+$ by $r_i(E_j)=\delta_{ij}$, and
\begin{equation*}
r_i(xy)=q^{(|y|, \alpha_i)}r_i(x)y+xr_i(y), \qquad x\in U_{q,|x|}^+, y\in U_{q,|y|}^+.
\end{equation*}
Here $U_{q, \beta}^+$ is the $Q^+$-graded subspace of $U_q^+$ with the weight $\beta$. The following result is from Lusztig \cite{L1}.

\begin{lemm} \label{L:comm-br}
1) If $x\in U_q^+$ and $r_i(x)=0$ for all $i\in\hat{I}$, then $x=0$.

2) One has $\{x\in U_q^+| r_i(x)=0\}=\{x\in U_q^+| T_i^{-1}(x)\in U_q^+\}$.
\end{lemm}

\section{Vertex subalgebras $U_q^{(i)}$}
There are three different lengths of real roots for the type of $A_{2n}^{(2)}$, which requires
a different treatment from other twisted types. For each $i\in I$ we construct
a copy of $U_q(\widehat{sl}_2)$ inside $U_q((X_N^{(r)})$ for $X_n\neq A_{2n}$. For the
case of $(A_{2n}^{(2)}, n)$ we will construct a subalgebra isomorphic to
$U_q(A_{2}^{(2)})$.

\subsection{Root system}
The root systems $\Delta$ of the twisted affine algebra $\widehat{\frak{g}}^{\sigma}$ is given by $\Delta=\Delta^{+}\cup -\Delta^{+}$ and $\Delta^{+}=\Delta_{>}\cup\Delta_{0}\cup\Delta_{<}$,
where $\Delta_{>}$ $\Delta_{0}$ and $\Delta_{<}$ are list as follows (see \cite{K}):
$$\Delta_{0}=\{k\delta|k>0\}\times I,$$
For the type $A_{2n}^{(2)}$
$$\Delta_{>}=\{k\delta+\alpha|\alpha\in \dot{\Delta}^+,k\geq 0\}\cup \{(2k+1)\delta+2\alpha|\alpha\in \dot{\Delta}^+,(\alpha,\,\alpha)=2,\,k\geq 0\},$$
$$\Delta_{<}=\{k\delta-\alpha|\alpha\in \dot{\Delta}^{+},k> 0\}\cup \{(2k+1)\delta-2\alpha|\alpha\in \dot{\Delta}^{+},(\alpha,\,\alpha)=2,\,k>0\},$$
for other types
$$\Delta_{>}=\{k d_{\alpha}\delta+\alpha|\alpha\in \dot{\Delta}^+,\,k\geq 0\},$$
$$\Delta_{<}=\{k d_{\alpha}\delta-\alpha|\alpha\in \dot{\Delta}^+,\,k>0\},$$
where $\dot{\Delta}^{+}$ is the positive roots system of $\frak{g}_{_0}$ and $d_{\alpha}=\frac{(\alpha, \alpha)}{2}$.
Here we call $\Delta^{+}$ the set of  positive roots of $\widehat{\frak{g}}^{\sigma}$.

\subsection{Quantum root vectors}\, In order define the quantum root vectors, we review some notations
from Beck-Nakajima \cite{BN}. For $\omega_i\in P, \forall i\in I$, we choose $\tau_i\in \mathcal{T}$ such that $\omega_i\tau_i^{-1}\in W$. Choose a reduced expression of $\omega_i\tau_i^{-1}$ for each $i\in I$. We fix a reduced expression of $\omega_n\omega_{n-1}\cdots \omega_1$ as follows:
$\omega_n\omega_{n-1}\cdots \omega_1=s_{i_1}s_{i_2}\cdots s_{i_m}\tau$, where $\tau=\tau_n\cdots \tau_1$.

We define a doubly infinite sequence
$$\mathbf{h}=(\dots, i_{-1}, i_0, i_1, \dots),$$
by setting $i_{l+m}=\tau(i_l)$ for $l\in \mathbb{Z}$.

Then we have $$\Delta_{>}=\{\alpha_{i_0},\, s_{i_0}(\alpha_{i_{-1}}),\, s_{i_0}s_{i_{-1}}(\alpha_{i_{-2}}),\,\dots \},$$
$$\Delta_{<}=\{\alpha_{i_1},\, s_{i_1}(\alpha_{i_{2}}),\, s_{i_1}s_{i_{2}}(\alpha_{i_3}),\,\dots \}.$$

Set $$\beta_k=\begin{cases}
s_{i_{0}}s_{i_{-1}}\cdots s_{i_{k+1}}(\alpha_{i_k}), \quad \hbox{if} \quad k\leq 0;\vspace{3pt}\\
s_{i_{1}}s_{i_{2}}\cdots s_{i_{k-1}}(\alpha_{i_k}), \qquad \hbox{if} \quad k>0.
\end{cases}$$
Define a total order on $\Delta^+$ by setting
$$\beta_0<\beta_{-1}<\beta_{-2}<\dots < \delta^{(1)}< \dots <\delta^{(n)}< 2\delta^{(1)}<\dots < \beta_2<\beta_1,$$
where $k\delta^{(i)}$ denotes $(k\delta, i)\in\Delta_0$.

The root vectors for each element of $\Delta_{>}\cup \Delta_{<}$ can be defined as follows:

$$E_{\beta_k}=\begin{cases}
T_{i_{0}}^{-1}T_{i_{-1}}^{-1}\cdots T_{i_{k+1}}^{-1}(E_{\alpha_{i_k}}), \quad \hbox{if} \quad k\leq 0;\vspace{3pt}\\
T_{i_{1}}T_{i_{2}}\cdots T_{i_{k-1}}(E_{\alpha_{i_k}}), \qquad \hbox{if} \quad k>0.
\end{cases}$$
It follows from \cite{L1} that the elements $E_{\beta_{k}}\in U_q^+$.

\begin{remark}\label{r1}\, For a positive root $\alpha=\omega_1(\alpha_i)$ in the root system $\Delta_{>}\cup\Delta_{<}$ of $U_q(\hat{\frak{g}}^{\sigma})$, there exists another presentation $\alpha=\omega_2(\alpha_j)$ where
$\omega_i, \omega_2$ are in the Weyl group $W$ and $\alpha_i, \alpha_j$ are simple roots. Then the quantum root vector $E_{\alpha}$ can be defined by two ways:
\begin{eqnarray*}
E_{\alpha}&\doteq &T_{\omega_1}(E_i)\\
E'_{\alpha}&\doteq &T_{\omega_2}(E_j)
\end{eqnarray*}

Actually, the two definitions agree up to a constant, because there exists $\omega\in W_0$ such that $\alpha_j=\omega(\alpha_i)$, which means that $\omega_1=\omega_2\omega$ and $l(\omega_1)=l(\omega_2)+l(\omega)$, then we have $$T_{\omega_1}(E_i)=T_{\omega_2\omega}(E_i)=c_1(q)T_{\omega_2}T_{\omega}(E_i)=c_2(q)T_{\omega_2}(E_j),$$
where $c_i(q)$ are functions in $q$.
\end{remark} \smallskip

Furthermore, if $\beta_k=\alpha_i$ for some $i\in I$, then $E_{\beta_{k}}=c(q)E_{\alpha_i}$ for
a function $c(q)$. As usual, define $F_{\beta}=\Phi(E_{\beta})$ for all $\beta\in\Delta_{>}\cup \Delta_{<}$.

Now we introduce the root vectors of Drinfeld generators \cite{BCP}.

\begin{eqnarray*}
&&E_{kp_i\delta+\alpha_i}=T_{\omega_i}^{-k}(E_i),  \quad \hbox{for}\quad k\geq 0,\\
&&E_{kp_i\delta-\alpha_i}=T_{\omega_i}^{k}T_i^{-1}(E_i), \quad \hbox{for}\quad k>0.
\end{eqnarray*}

Note that the definitions are not so different from those of non twisted cases up to some slight adjustments because of the difference between their root systems.

\subsection{Vertex subalgebra $U_q^{(i)}$}
\, Let subalgebra $U_q^{(i)}$ of twisted quantum affine algebra $U_q(\widehat{\frak{g}}^{\sigma})$ be generated by $E_i,\, T_{\omega_i}T_i^{-1}(E_i),\, F_i,\,T_{\omega_i}T_i^{-1}(F_i),\, K_i^{\pm1},\, T_{\omega_i}T_i^{-1}(K_i^{\pm 1})$ and $\gamma^{\pm \frac{1}{2}}$ for $i\in I$.

It is clear that the vertex subalgebra $U_q^{(i)}$ is stable under the actions of $\Phi,\, T_i,$ and $T_{\omega_i}$. Additionally, $U_q^{(i)}$ is pointwise fixed by $T_{\omega_j}$ for $j\neq i$.

We list the following result which is already proved in \cite{L2} (also see \cite{B}).
\begin{lemm} Let $i\neq j\in I$, one has that
 \begin{eqnarray*}
(1)&& \quad \ell(\omega_i s_i\omega_i)=2\ell(\omega_i)-1, \\
(2)&& T_i^{-1} T_{\omega_i}T_i^{-1}=T_{\omega_i}^{-1}\prod_{j\neq i}T_{\omega_j}^{-\frac{p_ia_{ij}}{p_j}}, \\
(3)&& T_{\omega_i}T_i^{-1}(E_j)=T_{\omega_j}^{-\frac{p_ia_{ij}}{p_j}}T_i(E_j).
\end{eqnarray*}
\end{lemm}

The following statements are based on the construction of \cite{B}. 
\begin{lemm} We have
 \begin{eqnarray*}
(1)&& [\,T_{\omega_i}T_i^{-1}(E_i),\,F_i\,]=0, \\
(2)&& E_i T_{\omega_i}^{-1}(E_i)-q_i^2 T_{\omega_i}^{-1}(E_i)E_i=0, \\
(3)&& T_{\omega_i}(K_i^{\pm1})=\gamma K_i^{\pm1}, \\
(4)&& K_i T_{\omega_i}^{-1}(E_i)=q_i^2 T_{\omega_i}^{-1}(E_i)K_i.
\end{eqnarray*}
\end{lemm}

\begin{prop}\, For $i\in I$, there exists an algebra isomorphism $\varphi_i: \, U_q(A_1^{(1)})\rightarrow U_q^{(i)}$,
defined by
\begin{eqnarray*}
& \varphi_i(E_1)=E_i,\quad &\varphi_i(E_0)=T_{\omega_i}T_i^{-1}(E_i),\\
& \varphi_i(F_1)=F_i,\quad  &\varphi_i(F_0)=T_{\omega_i}T_i^{-1}(F_i),\\
& \varphi_i(K_1^{\pm1})=K_i^{\pm1},\quad &\varphi_i(K_0^{\pm1})=T_{\omega_i}T_i^{-1}(K_i^{\pm1}),\\
& \varphi_i(q)=q_i, \quad  &\varphi_i(\gamma^{\pm \frac{1}{2}})=\gamma^{\pm \frac{1}{2}}.
\end{eqnarray*}
\end{prop}

\subsection{Relations of imaginary root vectors}
 Now we define the positive imaginary root vectors. For $k>0$ and $i \in I$ let $$\bar{\psi}_{i}(p_ik)=E_{kp_i\delta-\alpha_i}E_i-q_i^{-2}E_iE_{kp_i\delta-\alpha_i},$$
then we use $\bar{\varphi}_i(p_ik)$ to denote $\Phi(\bar{\psi}_i(p_ik))$.

Define the elements $E_{i\,kp_i\delta}\in U_q^+$ by the functional equation

$$exp\big((q_i-q_i^{-1})\sum_{k>0}E_{i,\,kp_i\delta}u^k \big)=1+(q_i-q_i^{-1})\sum_{k>0}\bar{\psi}_{i}(p_ik)u^k.$$

Similarly, us introduce $F_{i,\,kp_i\delta}=\Phi(E_{i,\,kp_i\delta})$ for $k>0$.
Then we have the following lemma

\begin{lemm}\, For $i,\,j\in I, k,\,l\geq 0$, we have
$$[\,E_{i,\,kp_i\delta},\,E_{j,\, lp_j\delta}\,]=0.$$
\end{lemm}

Define a map $o: I\rightarrow \{\pm 1\}$ such that $o(i)o(j)=-1$ for $a_{ij}<0$.
For $k>0$ and $i\in I$, let

$$a_i(p_ik)=o(i)^k\gamma^{-\frac{k}{2}}E_{i,\,kp_i\delta},$$
$${\psi}_{i}(p_ik)=o(i)^k (q_i-q_i^{-1})\gamma^{-\frac{k}{2}}K_i\bar{\psi}_{i}(p_ik).$$
\smallskip

Similarly, we can define $a_i(-p_ik)$ and ${\varphi}_i(-p_ik)$ for $k>0$.

\medskip

The following was given in \cite{Da2}, which is an application of Beck's work \cite{B}.
\begin{prop}\, For $i,\,j\in I$, for $k>0,\,l>0$, the following relations hold:
\begin{eqnarray*}
&(1)&\, [\,{\psi}_{i}(p_ik),\,{\psi}_{i}(p_il) \,]=0=[\,a_{i}(p_ik),\,a_{i}(p_il) \,],\\
&(2)&\, T_{\omega_j}\big({\psi}_{i}(p_ik)\big)={\psi}_{i}(p_ik),\quad T_{\omega_j}\big(a_{i}(p_ik)\big)=a_{i}(p_ik),\\
&(3)&\, [\,a_{i}(p_ik),\,a_{i}(-p_il) \,]=\delta_{k,l}\frac{[2r]_i}{k}\frac{\gamma^{k}-\gamma^{-k}}{q_i-q_i^{-1}},\\
&(4)&\, [\,a_{i}(p_ik),\,T_{\omega_i}^{-l}(E_i)\,]=\frac{[2r]_i}{k}T_{\omega_i}^{-(k+l)}(E_i).
\end{eqnarray*}
\end{prop}

\subsection{Relations between $U_q^{(i)}$ and $U_q^{(j)}$ }\,
We recall the commutation relations among the root vectors and remark that
the argument also works in the case of $A_{2n}^{(2)}$.

\begin{lemm}\, \cite{L2, B}
One has that for $i,\, j\in I$,
 \begin{eqnarray*}
&(1)& \quad T_{\omega_i}T_{\omega_j}=T_{\omega_j}T_{\omega_i},\\
&(2)& \quad T_{\omega_j}(E_i)=E_i,\quad \hbox{and} \quad T_{\omega_i}T_i=T_i T_{\omega_j} \,\, \hbox{for} \,\, i\neq j.
\end{eqnarray*}

\end{lemm}

We also let $\hat{T}_{\omega_i}=o(i){T}_{\omega_i}$. The following is immediate.

\begin{lemm}\, If $a_{ij}=0$ for $i, j\in I$,  then $[\,U_q^{(i)},\,U_q^{(j)}\,]=0.$
\end{lemm}


 \begin{lemm}\, For $i\neq j\in I$, and $k,\,l\in \mathbb{Z}$,
 $$[\,\hat{T}_{\omega_i}^{-k}(E_i),\,\hat{T}_{\omega_j}^{-l}(E_j)\,]=0.$$
\end{lemm}

If $i\neq j\in I$ such that $a_{ij}a_{ji}=1$, it is easy to see that
$\sigma(i)\neq i$ and $\sigma(j)\neq j$ and $d_i=p_i=1, d_j=p_j=1$.
Thus one has the following lemma.

\begin{lemm}\label{lemma2}\, If $i\neq j\in I$ such that $a_{ij}a_{ji}=1$, then for $k>1,\,l\in \mathbb{Z}$ one has following relations.
 \begin{eqnarray*}
&(1)&\, [\,{\psi}_{i}(1),\,E_j \,]=-\gamma^{-\frac{1}{2}}\hat{T}_{\omega_j}^{-1}(E_j),\\
&(2)&\, [\,{\psi}_{i}(k),\,\hat{T}_{\omega_j}^{-l}(E_j) \,]\\
&&\hspace{1.5cm}=\gamma^{-\frac{1}{2}} \Big(q \hat{T}_{\omega_j}^{-(l+1)}(E_j){\psi}_{i}(k-1)-q^{-1}{\psi}_{i}(k-1)\hat{T}_{\omega_j}^{-(l+1)}(E_j)\Big),\\
&(3)&\, [\,{\psi}_{i}(1),\,F_j \,]=\gamma^{-\frac{1}{2}}\hat{T}_{\omega_j}(F_j),\\
&(4)&\, [\,{\psi}_{i}(k),\,\hat{T}_{\omega_j}^{l}(F_j) \,]=\gamma^{\frac{1}{2}}\Big(q^{-1} \hat{T}_{\omega_j}^{(l+1)}(F_j){\psi}_{i}(k-1)-q{\psi}_{i}(k-1)\hat{T}_{\omega_j}^{(l+1)}(F_j)\Big).
\end{eqnarray*}
\end{lemm}
\vspace{6pt}

If $i\neq j\in I$ such that $a_{ij}=-r$ and $a_{ji}=-1$, it is easy
to see that $\sigma(i)\neq i$ and $\sigma(j)= j$ and $p_i=d_i=1,
d_j=r$. Then the next result follows.

\begin{lemm}\label{lemma3}\, Suppose $a_{ij}=-r$ and $a_{ji}=-1$ for $i\neq j\in I$. Then
for $k>r,\,l\in \mathbb{Z}$, we have
 \begin{eqnarray*}
&(1)&\, [\,{\psi}_{i}(1),\,\hat{T}_{\omega_j}^{-l}(E_j) \,]=-\gamma^{-\frac{1}{2}}[a_{ij}]_i \hat{T}_{\omega_j}^{-(l+1)}(E_j),\\
&(2)&\, [\,{\psi}_{i}(p_i k),\,\hat{T}_{\omega_j}^{-l}(E_j) \,]\\
&&\hspace{0.5cm}=\gamma^{-\frac{1}{2}} \Big(q_i^{-a_{ij}}\hat{T}_{\omega_j}^{-(l+1)}(E_j)  {\psi}_{i}(p_ik-r)-q_i^{a_{ij}}{\psi}_{i}(p_ik-r)\hat{T}_{\omega_j}^{-(l+1)}(E_j)\Big),\\
&(3)&\, [\,{\psi}_{i}(1),\,\hat{T}_{\omega_j}^{l}(F_j) \,]=\gamma^{-\frac{1}{2}}[a_{ij}]_i \hat{T}_{\omega_j}^{l+1}(F_j),\\
&(4)&\, [\,{\psi}_{i}(p_ik),\,\hat{T}_{\omega_j}^{l}(F_j) \,]\\
&&\hspace{0.5cm}=\gamma^{\frac{1}{2}}\Big(q_i^{a_{ij}} \hat{T}_{\omega_j}^{(l+1)}(F_j){\psi}_{i}(p_ik-r)
-q_{i}^{-a_{ij}}{\psi}_{i}(p_ik-r)\hat{T}_{\omega_j}^{(l+1)}(F_j)\Big).
\end{eqnarray*}

\end{lemm}

 The following statements follow directly from Lemma $\ref{lemma2}$ and $\ref{lemma3}$. Note that
 $d_ia_{ij}=\sum_{s=0}^{r-1}A_{i,\sigma^s(j)}$ for $i,\,j \in I$, which will be used later.

\begin{lemm}\, Let $i\neq j\in I$ such that $a_{ij}a_{ji}=1$. For $k>1,\,l\in \mathbb{Z}$ one has

 \begin{eqnarray*}
&(1)&\, [\,a_{i}(p_ik),\,\hat{T}_{\omega_j}^{-l}(E_j) \,]=\frac{[k a_{ij}]}{k}\hat{T}_{\omega_j}^{-(k+l)}(E_j)\\
&(2)&\, [\,a_{i}(p_ik),\,\hat{T}_{\omega_j}^{l}(F_j) \,]=-\frac{[k a_{ij}]}{k}\hat{T}_{\omega_j}^{(k+l)}(F_j),\\
&(3)&\, [\,a_{i}(k),\,a_{j}(-l) \,]=\delta_{k,l}\frac{[k a_{ij}]}{k}\frac{K_{k\delta}-K_{k\delta}^{-1}}{q_j-q_j^{-1}}.
\end{eqnarray*}

\end{lemm}

\begin{lemm}\label{lemma2}\, Let $i\neq j\in I$ such that $a_{ij}=-r$ and $a_{ji}=-1$. Then for $k>r,\,l\in \mathbb{Z}$,

 \begin{eqnarray*}
&(1)&\, [\,a_{i}(p_ik),\,\hat{T}_{\omega_j}^{-l}(E_j) \,]= \begin{cases}
\sum_{s=0}^{r-1}\frac{[k A_{i,\sigma^{s}(j)}]_i}{k}\hat{T}_{\omega_j}^{-(k+l)}(E_j), \quad \hbox{if} \quad p_j|k\vspace{6pt}\\
0, \qquad\qquad\qquad\qquad\qquad\qquad \hbox{otherwise}
\end{cases} \\
&(2)&\, [\,a_{i}(p_ik),\,\hat{T}_{\omega_j}^{l}(F_j) \,]=\begin{cases}
-\sum_{{s=0}}^{r-1}{\frac{[k A_{i,\sigma^{s}(j)}]_i}{k}}\hat{T}_{\omega_j}^{(k+l)}(F_j), \quad \hbox{if} \quad p_j|k \vspace{6pt}\\
0, \qquad\qquad\qquad\qquad\qquad\qquad \hbox{otherwise}
\end{cases},\\
&(3)&\, [\,a_{i}(p_ik),\,a_{j}(-p_jl) \,]=\delta_{kp_i,lp_j}\sum_{{s=0}}^{r-1}\frac{[k A_{i,\sigma^{s}(j)}]}{k}\frac{K_{kp_i\delta}-K_{kp_i\delta}^{-1}}{q_j-q_j^{-1}}.
\end{eqnarray*}

\end{lemm}

\subsection{Copies of $A_2^{(2)}$ in $A_{2n}^{(2)}$}
For the case of $A_{2n}^{(2)}$, the definition of the quantum root vectors differ from other twisted cases because of its slightly complicated root system. This situation has been discussed in details in \cite{Da2}, here we will review some definitions and results. The case of $A_2^{(2)}$ has been dealt with care in \cite{A}.

In this paragraph we discuss the remaining cases and show that for $i=n$, there exists a copy of
$U_q^{(i)}\simeq U_q(A_2^{(2)})$ in $A_{2n}^{(2)}$. 

Let us introduce the longest element of the Weyl group $W$:
$$w=s_0s_1s_2\cdots s_n,$$
then $w^{n-1}(\alpha_0)=\delta-2\alpha_n$, where $\delta=\alpha_0+2\alpha_1+\dots+2\alpha_n$.

In particular, define $E_{\delta-2\alpha_n}=T_{w}^{n-1}(E_0)$, and $F_{\delta-2\alpha_n}=\Phi(E_{\delta-2\alpha_n})$.

\begin{prop}\, Let  $U_q^{(n)}$ be the subalgebra generated by $E_n,\, F_n,\, K_n^{\pm 1},\, E_{\delta-2\alpha_n}$,
$F_{\delta-2\alpha_n},\,K_{\delta-2\alpha_n}$. There exists an
algebra isomorphism $\varphi_n: \, U_q(A_2^{(2)})\rightarrow
U_q^{(n)}$ defined as follows:
\begin{equation*}
\begin{split}
&\varphi_n(E_1)=E_n, \qquad\qquad   \varphi_n(F_1)=F_n, \qquad\qquad \varphi_n(K_1^{\pm 1})=F_n^{\pm 1}, \\
&\varphi_n(E_0)=E_{\delta-2\alpha_n}, \qquad \varphi_n(F_0)=F_{\delta-2\alpha_n},
\qquad \varphi_n(K_0^{\pm 1})=K_{\delta-2\alpha_n}^{\pm 1}.
\end{split}
\end{equation*}
\end{prop}

\section{Drinfeld realization for twisted cases}

\subsection{Drinfeld generators}\, In order to obtain the Drinfeld realization of twisted quantum affine algebras,
we introduce Drinfeld generators  as follows.

\begin{defi}\, For $k>0$, define

\begin{align*}
x_i^+(k)=\begin{cases}
\hat{T}_{\omega_i}^{-k}(E_i), \quad \hbox{if} \quad \sigma(i)\neq i \quad\hbox{or} \quad \sigma(i)=i \quad\hbox{and}\quad r|k \vspace{6pt}\\
0, \qquad \qquad\hbox{otherwise}
\end{cases},
\end{align*}
\begin{align*}
x_i^-(k)=\begin{cases}
-o(i)^k\gamma^{-k}K_iE_{k\delta-\alpha_i}, \quad \hbox{if} \quad \sigma(i)\neq i \quad\hbox{or} \quad \sigma(i)=i \quad\hbox{and}\quad r|k \vspace{6pt}\\
0, \qquad \qquad \hbox{otherwise}
\end{cases}.
\end{align*}
\end{defi}

\begin{defi}\, For  $k<0$, define

\begin{align*}
x_i^+(k)=\begin{cases}
-o(i)^kF_{-\alpha_i-k\delta}K_i^{-1}\gamma^k,\quad \hbox{if} \quad \sigma(i)\neq i \quad\hbox{or} \quad \sigma(i)=i \quad\hbox{and}\quad r|k;\vspace{6pt}\\
0, \qquad \qquad\hbox{otherwise},
\end{cases}
\end{align*}
\begin{align*}
x_i^-(k)=\begin{cases}
\hat{T}_{\omega_i}^{k}(F_i), \quad \hbox{if} \quad \sigma(i)\neq i \quad\hbox{or} \quad \sigma(i)=i \quad\hbox{and}\quad r|k;\vspace{6pt}\\
0, \qquad \qquad \hbox{otherwise},
\end{cases}
\end{align*}
\end{defi}

\begin{remark}\,From the above definitions it follows that $a_i(k)=0$  if $\sigma(i)=i$ and $k$ is not divisible
by $r$.
\end{remark}

For convenience we extend the indices from $I$ to $\{1,\,2,\, \dots, N\}$. For $i\in
\{n+1, n+2,\, \dots , N\}$ and $k\in\mathbb{Z}$, $l\in
\mathbb{Z}/\{0\}$, we define that
$$x_i^{\pm}(k)=\omega^{-k}x_{\sigma(i)}^{\pm}(k),\quad a_i(l)=\omega^{-l}a_{\sigma(i)}(l),\quad K_i=K_{\sigma(i)}.$$

\subsection{Drinfeld realization for twisted cases}
In previous sections we have prepared for the relations among Drinfeld generators. The complete
relations are given in the following theorem stated first in
\cite{Dr}. In the following we will set out to prove the remaining Serre relations using braid groups
and other techniques developed in \cite{ZJ, JZ}.

\begin{theo}\label{t1}{\, The twisted quantum affine algebra $U_q(\widehat{\frak{g}}^{\sigma})$ is generated by the elements $x_i^{\pm}(k), a_i(l),\, K_i^{\pm1}$ and $\gamma^{\frac{1}{2}}$, where $i\in \{1,\,2,\, \dots, N\}$, $k\in \mathbb{Z}$, and $l\in \mathbb{Z}/\{0\}$, satisfying the following relations:
\begin{eqnarray*}
          &(1)&x_i^{\pm}(k)=\omega^{-k}x_{\sigma(i)}^{\pm}(k),\quad a_i(l)=\omega^{-l}a_{\sigma(i)}(l),\quad K_i=K_{\sigma(i)},\\
          &(2)&\quad [\gamma^{\pm \frac{1}{2}},~u]=0    \ \ \ \mbox{for all} \,\, u\in U_q(\widehat{\frak{g}}^{\sigma}),\vspace{0.1cm}\\
        &(3)&\quad [a_i(k),a_j(l)]=\delta_{k+l,~0}\sum\limits_{s=0}^{r-1}
         \frac{[k A_{i,\sigma^s(j)}]_i}{k}\omega^{ks}
         \frac{\gamma^k-\gamma^{-k}}{q_j-q_j^{-1}},\vspace{0.1cm}\\
       &(4)&\quad [a_i(k),K_j^{\pm}]=0,\vspace{0.1cm}\\
       &(5)&\quad K_ix_j^{\pm}(k) K_i^{-1} =q^{\pm\sum\limits_{s=0}^{r-1} A_{i,\sigma^s(j)}}x_j^{\pm}(k),\vspace{0.1cm}\\
       &(6)&\quad [a_i(k),~x_j^{\pm}(l)]= \pm \sum\limits_{s=0}^{r-1}
         \frac{[k A_{i,\sigma^s(j)}]_i}{k}\omega^{ks}
         \gamma^{\mp \frac{|k|}{2}}x_j^{\pm}{(k+l)},\vspace{0.1cm}\\
    &(7)&\quad [x_i^+(k),x_j^-(l)]=\sum\limits_{s=0}^{r-1}
\frac{\delta_{\sigma^s(i),j}\omega^{ls}}{q_i-q_i^{-1}}
(\gamma^{\frac{k-l}{2}}\psi_i(k+l)-\gamma^{\frac{l-k}{2}}\varphi_i(k+l)).
\end{eqnarray*}
where $\psi_i(m)$ and $\varphi_i(m)$  $(m\in \mathbb{Z}_{\geq 0})$ are defined by\\
$$\quad\sum\limits_{m=0}^{\infty}\psi_i(m)z^{-m} =K_i exp((q_i-q_i^{-1})\sum\limits_{k=1}^{\infty}
a_i(k)z^{-k}), $$
$$\sum\limits_{m=0}^{\infty}\varphi_i(-m)z^m =K_i^{-1} exp(-(q_i-q_i^{-1})\sum\limits_{k=1}^{\infty}
a_i(-k)z^k). $$
\begin{eqnarray*}
 &(8)&\quad \prod\limits_{s=0}^{r-1}(z-\omega^s
q^{\pm A_{i,\sigma^s(j)}}w)x_i^{\pm}(z)x_j^{\pm}(w)=
\prod
\limits_{s=0}^{r-1}(z q^{\pm A_{i,\sigma^s(j)}}-\omega^sw)
       x_j^{\pm}(w)x_i^{\pm}(z).\vspace{0.1cm}\\
&(9)&\,  Sym_{z_1,z_2}P_{ij}^{\pm}(z_1,z_2)\sum\limits_{s=0}^{2}(-1)^s\begin{bmatrix}2\\s\end{bmatrix}_{q^{d_{ij}}}
x_i^{\pm}(z_1)\cdots
x_i^{\pm}(z_s)x_j^{\pm}(w)x_i^{\pm}(z_{s+1})\cdots x_j^{\pm}(z_2)=0,\\
&&\hspace{5.8cm}\hbox{for} ~ A_{ij}=-1, \, \sigma(i)\neq j,\\
&(10)&\, Sym_{z_1,z_2,z_3}[(q^{\mp\frac{3r}{4}}z_1-(q^{\frac{r}{4}}+q^{-\frac{r}{4}})z_2+
q^{\pm\frac{3r}{4}}z_3)
x_i^{\pm}(z_1)x_i^{\pm}(z_2)x_i^{\pm}(z_3)]=0. \\
&& \hspace{5.8cm} \hbox{for} ~A_{i,\sigma(i)}=-1.
\end{eqnarray*}
 where  Sym means
the symmetrization over $z_i$, $x_i^{\pm}(z) = \sum_{k \in \mathbb{Z}}x_i^{\pm}(k) z^{-k}$,
 $P_{ij}^{\pm}(z,w)$
and $d_{ij}$ are defined as follows:\\
\begin{eqnarray*}
&&\hbox{If}\,\, \sigma(i)=i,\, \hbox{then}\,\,  P_{ij}^{\pm}(z,w)=1,\,  d_{ij}=r/2,\\
&&\hbox{If}\,\, A_{i,\sigma(i)}=0, \, \sigma(j)=j, \, \hbox{then\, } P_{ij}^{\pm}(z,w)=
\frac{z^rq^{\pm2r}-w^r}{zq^{\pm2}-w},\,  d_{ij}=r, \\
&&\hbox{If}\,\,  A_{i,\sigma(i)}=0, \, \sigma(j)\neq
j,\, \hbox{then}\,\, P_{ij}^{\pm}(z,w)=1,\, d_{ij}=1/2, \\
&&\hbox{If}\,\,   A_{i,\sigma(i)}=-1,\, \hbox{then}\,\, P_{ij}^{\pm}(z,w)=
zq^{\pm r/2}+w,\,  d_{ij}=r/2.
\end{eqnarray*}}
\end{theo}

\subsection{Proof of the main theorem}\, We need to verify that the above Drinfeld generators
 $x_i^{\pm}(k), a_i(l),\, K_i^{\pm1}$ satisfy all relations $(1)-(10)$.
The relations $(1)-(7)$ are already checked in the previous paragraphs.
We are going to show the last three relations.

We first proceed to check relation $(8)$.

\begin{prop}\, For all $i, j\in I$ one has that
$$ \prod\limits_{s=0}^{r-1}(z-\omega^s
q^{\pm A_{i,\sigma^s(j)}}w)x_i^{\pm}(z)x_j^{\pm}(w)=
\prod
\limits_{s=0}^{r-1}(z q^{\pm A_{i,\sigma^s(j)}}-\omega^sw)
       x_j^{\pm}(w)x_i^{\pm}(z),$$
where $x_i^{\pm}(z) = \sum_{k \in \mathbb{Z}}x_i^{\pm}(k) z^{-k}$.
\end{prop}

\, The relation holds if $A_{ij}=0$, so we only consider the case of $A_{ij}\neq 0$.  The proof is divided into several cases.\\

{\bf Case (a): $i=j$.} The required relations are generating functions of the following component relations:

$$x_i^{\pm}(k+p_i)x_i^{\pm}(l)-q_i^{\pm 2}x_i^{\pm}(l)x_i^{\pm}(k+p_i)=
q_i^{\pm 2}x_i^{\pm}(k)x_i^{\pm}(l+p_i)-x_i^{\pm}(l+p_i)x_i^{\pm}(k).$$

On the other hand the following relations hold in relation in $U_q^{(i)}$
similar as in the untwisted cases.
\begin{eqnarray*}
&&\,  T_{\omega_i}^{-k}(E_i)E_i-q_i^2E_i T_{\omega_i}^{-k}(E_i)=q_i^{2}T_{\omega_i}^{-k+1}(E_i)T_{\omega_i}(E_i)-
T_{\omega_i}(E_i)T_{\omega_i}^{-k+1}(E_i).
\end{eqnarray*}
Hence the required relation follows by recalling the definition of $x_i^{+}(k)$.\\

{\bf Case (b): $i\neq j$  such that $A_{ij}\neq 0$.}

First for $i, j$ we define
$$E_{ij}=-E_iE_j+q_i^{A_{ij}}E_jE_i.$$

\begin{lemm}\, For $i\neq j\in I$ such that $A_{ij}\leqslant 0$, and $k\in \mathbb{Z}$,
  \begin{eqnarray*}
&& -T_{\omega_i}^{-k}(E_i)T_{\omega_j}^{-l}(E_j)+q_i^{A_{ij}}T_{\omega_j}^{-l}(E_j)T_{\omega_i}^{-k}(E_i)\\
&=&q_i^{A_{ij}}T_{\omega_i}^{-(k-p_j)}(E_i)T_{\omega_j}^{-(l+p_i)}(E_j)
- T_{\omega_j}^{-(l+p_i)}(E_j)T_{\omega_i}^{-(k-p_j)}(E_i).
\end{eqnarray*}
\end{lemm}
\begin{proof}
Applying $T_{\omega_i}^{-k}$ and $T_{\omega_j}^{-l}$ to $E_{ij}$ and invoking
Lemma 3.7, we can pull out the action of $T_{\omega_j}$ to arrive at
$$
-T_{\omega_i}^{-k}(E_i)E_j+q_i^{A_{ij}}E_jT_{\omega_i}^{-k}(E_i)
=q_i^{A_{ij}}T_{\omega_i}^{-(k-p_j)}(E_i)T_{\omega_j}^{-p_i}(E_j)
- T_{\omega_j}^{-p_i}(E_j)T_{\omega_i}^{-(k-p_j)}(E_i).
$$
which was essentially proved by Beck \cite{B} since $d_jp_i=d_ip_j$.
\end{proof}

The following well-known fact will be used to prove the remaining relations.
\begin{lemm}\label{lemma1}\, If $A\in U_q(\hat{\frak{g}}^{\sigma})^+$ and $[\,A,\, F_k\,]=0 \, \forall k\in \hat{I}$, then $A=0$.
\end{lemm}

We concentrate mainly on the relation $(9)$ and divide it into
four cases. The Serre relation in the case of
$P_{ij}^{\pm}(z_1,z_2)=1$ can be derived from that of non twisted
case.  
Moreover, the proof will explain why the Serre relation with the
lower power works by the action of diagram automorphism $\sigma$ in
the twisted case.

\begin{prop}\label{prop1}\, For $A_{ij}=-1$ and $\sigma(i)\neq j$, we have:
\begin{eqnarray*}
&&Sym_{z_1,z_2}P_{ij}^{\pm}(z_1,z_2)\sum\limits_{s=0}^{2}(-1)^s
x_i^{\pm}(z_1)\cdots
x_i^{\pm}(z_s)x_j^{\pm}(w)x_i^{\pm}(z_{s+1})\cdots x_j^{\pm}(z_2)=0
\end{eqnarray*}
\end{prop}

\begin{proof}\, This is proved case by case.

{\bf Case (i): $A_{ij}=-1$ and $\sigma(i)=i$.}

In this case  $P_{ij}^{\pm}(z_1,z_2)=1$ and $d_{ij}=r$, it is also
clear that $d_i=r$, then the relation is exactly like the Serre relation
in the non-twisted case. For completeness we provide a proof for this Serre
relation. i.e. we will show that for any integers $k_1, k_2, l$
$$Sym_{k_1,k_2}\Big(x_j^{+}(l)x_i^{+}(k_1)x_i^{+}(k_2)-[2]_ix_i^{+}(k_1)x_j^{+}(l)x_i^{+}(k_2) +x_i^{+}(k_1)x_i^{+}(k_2)x_j^{+}(l)\Big)=0.$$
Note that $x_j^+(l)=T_{\omega_j}^{-l}E_j$. Lemma 3.7 says that one can pull out any factor
of $T_{\omega_j}$ or common factors of $T_{\omega_i}$ from the left-hand side (LHS).
This means that for any natural number $t$ the following relation is equivalent to the Serre relation.
  $$\Big(E_j E_i\hat{T}_{\omega_i}^{-t}(E_i)-[2]_i E_i E_j\hat{T}_{\omega_i}^{-t}(E_i) +E_i\hat{T}_{\omega_i}^{-t}(E_i)E_j\Big)+\Big(E_i\leftrightarrow \hat{T}_{\omega_i}^{-t}(E_i)\Big)=0.$$

We prove this last relation by induction on $t$. First note that when $t=0$, the relation is essentially
the relation $(R5)$.
We assume that the above relation holds for $\leq t-1$. The remark above further says once we have made
the inductive assumption then all Serre relations with
$|k_1-k_2|\leq t-1$ and arbitrary $l$ are also assumed to be true.
Using relation $(6)$ in Theorem
\ref{t1} yields
$$\hat{T}_{\omega_i}^{-1}(E_i)=\frac{\gamma^{\frac{1}{2}}}{[2]_i}[\,a_i(1),\,E_i\,].  \qquad\quad (*)$$
Plugging this into LHS of the Serre relation we get that
\begin{eqnarray*}
&&\frac{\gamma^{\frac{1}{2}}}{[2]_i}\Big(E_j E_i\hat{T}_{\omega_i}^{-t+1}\big([\,a_i(1),\,E_i\,]\big)-[2]_i E_i E_j\hat{T}_{\omega_i}^{-t+1}\big([\,a_i(1),\,E_i\,]\big) \\ &&\hspace{1.5cm}+E_i\hat{T}_{\omega_i}^{-t+1}\big([\,a_i(1),\,E_i\,]\big)E_j\Big)+\Big(E_i\leftrightarrow \hat{T}_{\omega_i}^{-t+1}(E_i)\Big)\\
&=&\frac{\gamma^{\frac{1}{2}}}{[2]_i}\Big(E_j E_i[\,a_i(1),\,\hat{T}_{\omega_i}^{-t+1}(E_i)\,]-[2]_i E_i E_j[\,a_i(1),\,\hat{T}_{\omega_i}^{-t+1}(E_i)\,] \\ &&\hspace{1.5cm}+E_i[\,a_i(1),\,\hat{T}_{\omega_i}^{-t+1}(E_i)\,]E_j\Big)+\Big(E_i\leftrightarrow \hat{T}_{\omega_i}^{-t+1}(E_i)\Big).
\end{eqnarray*}
Then we repeatedly use $(*)$ to move $a_i(1)$ to the extreme left
to get an expression of the form
$$
a_i(1)\Big(E_j E_i\hat{T}_{\omega_i}^{-t+1}(E_i)-[2]_i E_i E_j\hat{T}_{\omega_i}^{-t+1}(E_i) +E_i\hat{T}_{\omega_i}^{-t+1}(E_i)E_j\Big)+\cdots
$$
where $\cdots$ only involves with LHS of Serre relations with $t-2$. So the whole
expression is zero by the inductive assumption.
Thus we have finished the proof of the Serre relation $(9)$ in this case.

{\bf Case (ii): $A_{ij}=-1$ and $A_{i,\sigma(i)}=0, \sigma(j)=j$.}

 For $r=2$, without loss generality
 we take $A_{2n-1}^{(2)}$ for an example, the other cases are treated similarly.
 In this case we only need to consider the situation when $i=n-1$ and $j=n$,
 then  $P_{ij}^{\pm}(z_1,z_2)=z_1q^{\pm 2}+z_2$ and $d_{ij}=2$. So we need to prove the following relations.
\begin{eqnarray*}
&&q^2\Big(x_j^+(l)x_i^+(k+1)x_i^+(k)-[2]_{q^2}x_i^+(k+1)x_j^+(l)x_i^+(k)+x_i^+(k+1)x_i^+1(k)x_j^+(l)\Big)\\
&&+\Big(x_j^+(l)x_i^+(k)x_i^+(k+1)-[2]_{q^2}x_i^+(k)x_j^+(l)x_i^+(k+1)+x_i^+(k)x_i^+(k+1)x_j^+(l)\Big)\\
&&=0.
\end{eqnarray*}

Using the definition of $x_i^+(k)$ and collecting the action of $\hat{T}_{\omega_i}^{-k}\hat{T}_{\omega_j}^{-l}$,
 we are left to show that

\begin{eqnarray*}
&X=&q^2\Big(E_j \hat{T}_{\omega_i}^{-1}(E_i)E_i-[2]_{q^2} \hat{T}_{\omega_i}^{-1}(E_i)E_j E_i+ \hat{T}_{\omega_i}^{-1}(E_i)E_i E_j \Big)\\
&&+\Big(E_j E_i\hat{T}_{\omega_i}^{-1}(E_i)-[2]_{q^2}E_i E_j \hat{T}_{\omega_i}^{-1}(E_i)+ E_i\hat{T}_{\omega_i}^{-1}(E_i) E_j\Big)=0.
\end{eqnarray*}

To see this we use Lemma $\ref{lemma1}$
and compute all the commutators $[X, F_k]=0$ for $k\in\hat{I}$. First
we consider the case of $k=0$.
Note that $[E_i,\, F_0]=0$ whenever $i\neq 0$.
On the other hand we claim that $[\,\hat{T}_{\omega_i}^{-1}(E_i),\,F_0\,]=0$. To see this we check
that $r_0(\hat{T}_{\omega_i}^{-1}(E_i))=0$ by using the twisted derivation $r_0$.
By Lemma \ref{L:comm-br} the last equation is equivalent to $T_0^{-1}\hat{T}_{\omega_i}^{-1}(E_i)\in U_q^{+}$, which
can be easily seen as $s_0(\delta+\alpha_i)\in \Delta_+$ when $\mathfrak g_0\neq \mathfrak{sl}_2$. Consequently it implies that $[\,X,\, F_0\,]=0$.

Since $X$ is expressed by $E_i$ and $E_j$,  $[X, F_k]=0$ for any $k\neq i,j$.
Using Drinfeld relation $(7)$, we compute that
\begin{eqnarray*}
& &[\,X,\,F_j\,]\\
&=&q^2\Big(\underbrace{[\,E_j,\,F_j\,]}\hat{T}_{\omega_i}^{-1}(E_i)E_i-[2]_{q^2} \hat{T}_{\omega_i}^{-1}(E_i)\underbrace{[\,E_j,\,F_j\,]} E_i+ \hat{T}_{\omega_i}^{-1}(E_i)E_i \underbrace{[\,E_j,\,F_j\,]} \Big)\\
&&+\Big(\underbrace{[\,E_j,F_j\,]} E_i\hat{T}_{\omega_i}^{-1}(E_i)-[2]_{q^2}E_i\underbrace{[\,E_j,\,F_j\,]} \hat{T}_{\omega_i}^{-1}(E_i)+ E_i\hat{T}_{\omega_i}^{-1}(E_i)\underbrace{[\,E_j,\,F_j\,]}\Big)\\
&=&\frac{q^2}{q_i-q_i^{-1}}\Big(\underbrace{\big(1-[2]_{q^2}q^{-2}+q^{-4}\big)}K_i
+\underbrace{\big(1-[2]_{q^2}q^{2}+q^{4}\big)}K_i^{-1}\Big)\hat{T}_{\omega_i}^{-1}(E_i)E_i\\
&&+\frac{1}{q_i-q_i^{-1}}\Big(\underbrace{\big(1-[2]_{q^2}q^{-2}+q^{-4}\big)}K_i
+\underbrace{\big(1-[2]_{q^2}q^{2}+q^{4}\big)}K_i^{-1}\Big)E_i\hat{T}_{\omega_i}^{-1}(E_i)\\
&=&0,
\end{eqnarray*}
where  Drinfeld relation $(5)$ has been used.

\begin{eqnarray*}
& &[\,X,\,F_i\,]\\
&=&q^2\Big(E_j\underbrace{[\,\hat{T}_{\omega_i}^{-1}(E_i),\,F_i\,]}E_i-[2]_{q^2} \underbrace{[\,\hat{T}_{\omega_i}^{-1}(E_i),\,F_i\,]}E_j E_i+ \underbrace{[\,\hat{T}_{\omega_i}^{-1}(E_i),\,F_i\,]} E_i E_j\\
&&+E_j\hat{T}_{\omega_i}^{-1}(E_i)\underbrace{[\,E_i,\,F_i\,]}-[2]_{q^2} \hat{T}_{\omega_i}^{-1}(E_i)E_j\underbrace{[\,E_i,\,F_i\,]}+ \hat{T}_{\omega_i}^{-1}(E_i)\underbrace{[\,E_i,\,F_i\,]} E_j \Big)\\
&&+\Big(E_j\underbrace{[\,E_i,F_i\,]}\hat{T}_{\omega_i}^{-1}(E_i)-[2]_{q^2}\underbrace{[\,E_i,\,F_i\,]}E_j \hat{T}_{\omega_i}^{-1}(E_i)+ \underbrace{[\,E_i,\,F_i\,]}\hat{T}_{\omega_i}^{-1}(E_i)E_j\\
&&+E_j E_i\underbrace{[\,\hat{T}_{\omega_i}^{-1}(E_i),F_i\,]}-
[2]_{q^2}E_i E_j\underbrace{[\,\hat{T}_{\omega_i}^{-1}(E_i),\,F_i\,]} + E_i\underbrace{[\,\hat{T}_{\omega_i}^{-1}(E_i),\,F_i\,]}E_j\Big)
\end{eqnarray*}
where we have used the Drinfeld relation $(7)$, $(5)$ and $(6)$ for the last step. Collecting common terms, we arrive at
\begin{eqnarray*}
& &[\,X,\,F_i\,]\\
&=&\Big(\gamma^{\frac{1}{2}}\underbrace{\big(q^4-q^2[2]_{q^2}+1\big)}E_j E_i a_i(1)K_i+\big(q^4-q^2[2]_{q^2}\big)[2]E_j \hat{T}_{\omega_i}^{-1}(E_i)K_i\Big)\\
&&+\Big(\gamma^{\frac{1}{2}}\underbrace{\big(q^2-[2]_{q^2}+q^{-2}\big)}E_i E_j a_i(1)K_i
+q^2[2] \hat{T}_{\omega_i}^{-1}(E_i)E_j K_i\Big)\\
&&+\frac{1}{q_i-q_i^{-1}}\Big(\underbrace{\big(q^2-[2]_{q^2}+q^{-2}\big)}E_j \hat{T}_{\omega_i}^{-1}K_i^{-1}
+\big(2q^2-[2]_{q^2}\big)E_j \hat{T}_{\omega_i}^{-1}K_i\Big) \\
&&+\frac{1}{q_i-q_i^{-1}}\Big(\underbrace{\big(1-[2]_{q^2}q^2+q^{4}\big)}\hat{T}_{\omega_i}^{-1}E_j K_i^{-1}
+\big(2-[2]_{q^2}q^2\big) \hat{T}_{\omega_i}^{-1}E_j K_i\Big)\\
&=&\big([2]-[2]\big)E_j \hat{T}_{\omega_i}^{-1}(E_i)K_i+\big(q^2[2]-q^2[2]\big) \hat{T}_{\omega_i}^{-1}(E_i)E_j K_i\\
&=&0.
\end{eqnarray*}

For $r=3$, the exceptional type $D_4^{(3)}$ should also be checked. In this
case we know that $P_{ij}^{\pm}(z_1,z_2)=z_1^2q^{\pm 4}+z_1z_2q^{\pm
2}+z_2^2$ and $d_{ij}=3$. More specifically we have
in this case $i=2$, $j=1$ and
$d_i=2, d_j=1$, and the relation is
reduced to the following equivalent one:
\begin{eqnarray*}
&&q^4\Big(x_j^+(l)x_i^+(k+2)x_i^+(k)-[2]_{q^3}x_i^+(k+2)x_j^+(l)x_i^+(k)+x_i^+(k+2)x_i^+1(k)x_j^+(l)\Big)\\
&&+q^2\Big(x_j^+(l)(x_i^+(k+1))^2-[2]_{q^3}x_i^+(k+1)x_j^+(l)x_i^+(k+1)+(x_i^+(k+1))^2x_j^+(l)\Big)\\
&&+\Big(x_j^+(l)x_i^+(k)x_i^+(k+2)-[2]_{q^3}x_i^+(k)x_j^+(l)x_i^+(k+2)+x_i^+(k)x_i^+(k+2)x_j^+(l)\Big)\\
&&=0.
\end{eqnarray*}

By definition of $x_i^+(k)$ and collecting the action of
$\hat{T}_{\omega_i}^{-k}\hat{T}_{\omega_j}^{-l}$, we can rewrite the LHS of the above
relation as follows.

\begin{eqnarray*}
&Y=&q^4\Big(E_j \hat{T}_{\omega_i}^{-2}(E_i)E_i-[2]_{q^3} \hat{T}_{\omega_i}^{-2}(E_i)E_j E_i+ \hat{T}_{\omega_i}^{-2}(E_i)E_i E_j \Big)\\
&&+q^2\Big(E_j \hat{T}_{\omega_i}^{-1}(E_i^2)-[2]_{q^3} \hat{T}_{\omega_i}^{-1}(E_i)E_j \hat{T}_{\omega_i}^{-1}(E_i)+ \hat{T}_{\omega_i}^{-1}(E_i^2)E_j \Big)\\
&&+\Big(E_j E_i\hat{T}_{\omega_i}^{-2}(E_i)-[2]_{q^2}E_i E_j \hat{T}_{\omega_i}^{-2}(E_i)+ E_i\hat{T}_{\omega_i}^{-2}(E_i) E_j\Big)
\end{eqnarray*}

Let us use Lemma 4.7 to show that $Y=0$ by checking that $[\,Y, F_k\,]=0$ for all $k\in
\hat{I}$. When $k=0$ is clear as $Y$ is only expressed in terms of $E_i$ and $E_j$
and $i, j\in I$. This also implies that $[Y, F_k]=0$ for $k\neq i$ or $j$.


For $k=j$, we use Drinfeld relation (4) and the commutation relation
$[\,E_j,\,F_j\,]=\frac{K_j-K_j^{-1}}{q_j-q_j^{-1}}$ to get that
\begin{eqnarray*}
& &[\,Y,\,F_j\,]\\
&=&q^4\Big([\,E_j,\,F_j\,] \hat{T}_{\omega_i}^{-2}(E_i)E_i-[2]_{q^3} \hat{T}_{\omega_i}^{-2}(E_i)[\,E_j,\,F_j\,] E_i+ \hat{T}_{\omega_i}^{-2}(E_i)E_i [\,E_j,\,F_j\,] \Big)\\
&&+q^2\Big([\,E_j,\,F_j\,]\hat{T}_{\omega_i}^{-1}(E_i^2)-[2]_{q^3} \hat{T}_{\omega_i}^{-1}(E_i)[\,E_j,\,F_j\,] \hat{T}_{\omega_i}^{-1}(E_i)+ \hat{T}_{\omega_i}^{-1}(E_i^2)[\,E_j,\,F_j\,] \Big)\\
&&+\Big([\,E_j,\,F_j\,] E_i\hat{T}_{\omega_i}^{-2}(E_i)-[2]_{q^2}E_i [\,E_j,\,F_j\,] \hat{T}_{\omega_i}^{-2}(E_i)+ E_i\hat{T}_{\omega_i}^{-2}(E_i) [\,E_j,\,F_j\,]\Big)\\
&=&\frac{q^4}{q_j-q_j^{-1}}\Big(\underbrace{\big(1-[2]_{q^3}q^{3}+q^{6}\big)}K_j
+\underbrace{\big(1-[2]_{q^3}q^{-3}+q^{-6}\big)}K_j^{-1}\Big)\hat{T}_{\omega_i}^{-2}(E_i)E_i\\
&&+\frac{q^2}{q_j-q_j^{-1}}\Big(\underbrace{\big(1-[2]_{q^3}q^{3}+q^{6}\big)}K_i
+\underbrace{\big(1-[2]_{q^3}q^{-3}+q^{-6}\big)}K_j^{-1}\Big)\hat{T}_{\omega_i}^{-2}(E_i^2)\\
&&+\frac{1}{q_j-q_j^{-1}}\Big(\underbrace{\big(1-[2]_{q^3}q^{3}+q^{6}\big)}K_i
+\underbrace{\big(1-[2]_{q^3}q^{-3}+q^{-6}\big)}K_j^{-1}\Big)E_i\hat{T}_{\omega_i}^{-2}(E_i)\\
&=&0.
\end{eqnarray*}
Next we calculate that 
\begin{eqnarray*}
&&[\,Y,\,F_i\,]\\
&=&q^4\Big(E_j \underbrace{[\,\hat{T}_{\omega_i}^{-2}(E_i),\,F_i\,]}E_i-[2]_{q^3} \underbrace{[\,\hat{T}_{\omega_i}^{-2}(E_i),\,F_i\,]}E_j E_i+ \underbrace{[\,\hat{T}_{\omega_i}^{-2}(E_i),\,F_i\,]}E_i E_j \Big)\\
&&+q^4\Big(E_j \hat{T}_{\omega_i}^{-2}(E_i)\underbrace{[\,E_i,\,F_i\,]}-[2]_{q^3} \hat{T}_{\omega_i}^{-2}(E_i)E_j\underbrace{[\,E_i,\,F_i\,]}+ \hat{T}_{\omega_i}^{-2}(E_i)\underbrace{[\,E_i,\,F_i\,]} E_j \Big)\\
&&+q^2\Big(E_j\underbrace{[\,\hat{T}_{\omega_i}^{-1}(E_i),F_i\,]}\hat{T}_{\omega_i}^{-1}(E_i)-[2]_{q^3} \underbrace{[\,\hat{T}_{\omega_i}^{-1}(E_i),F_i\,]}E_j \hat{T}_{\omega_i}^{-1}(E_i)\\
&&\hspace{6.95cm}+ \underbrace{[\,\hat{T}_{\omega_i}^{-1}(E_i),F_i\,]}\hat{T}_{\omega_i}^{-1}(E_i)E_j \Big)\\
&&+q^2\Big(E_j\hat{T}_{\omega_i}^{-1}(E_i)\underbrace{[\,\hat{T}_{\omega_i}^{-1}(E_i),F_i\,]}
-[2]_{q^3}\hat{T}_{\omega_i}^{-1}(E_i) E_j \underbrace{[\,\hat{T}_{\omega_i}^{-1}(E_i),F_i\,]}\\
&&\hspace{6.95cm}+ \hat{T}_{\omega_i}^{-1}(E_i)\underbrace{[\,\hat{T}_{\omega_i}^{-1}(E_i),F_i\,]}E_j \Big)\\
&&+\Big(E_j\underbrace{[\,\hat{T}_{\omega_i}^{-2}(E_i),\,F_i\,]}E_i-[2]_{q^2}E_i E_j\underbrace{[\,\hat{T}_{\omega_i}^{-2}(E_i),\,F_i\,]}+ E_i\underbrace{[\,\hat{T}_{\omega_i}^{-2}(E_i),\,F_i\,]}E_j \Big)\\
&&+\Big(E_j\hat{T}_{\omega_i}^{-2}(E_i)\underbrace{[\,E_i,\,F_i\,]}
-[2]_{q^2}\underbrace{[\,E_i,\,F_i\,]} E_j\hat{T}_{\omega_i}^{-2}(E_i)+ \underbrace{[\,E_i,\,F_i\,]}\hat{T}_{\omega_i}^{-2}(E_i)E_j \Big)
\end{eqnarray*}
To show the above is actually zero, we collect similar terms into five summands.
\begin{eqnarray*}
&&\hbox{The first term}\\
&=&\gamma\Big(q^6 E_j \big(a_i(2)+\frac{q_i-q_i^{-1}}{2}(a_i(1))^2\big)E_i-q^3[2]_{q^3}
\big(a_i(2)+\frac{q_i-q_i^{-1}}{2}(a_i(1))^2\big)E_j E_i\\
&&\hspace{4.95cm}+ E_j E_i\big(a_i(2)+\frac{q_i-q_i^{-1}}{2}(a_i(1))^2\big)\Big)K_i\\
&=&\Big(-\frac{[4]}{2}E_j\hat{T}_{\omega_i}^{-2}(E_i)-\gamma^{\frac{1}{2}}\frac{[2]}{2}(q-q^{-1})E_j\big( a_i(1)\hat{T}_{\omega_i}^{-1}(E_i)+ \hat{T}_{\omega_i}^{-1}(E_i)a_i(1)\big)\Big)K_i
\end{eqnarray*}

\begin{eqnarray*}
&&\hbox{The second term}\\
=&&\gamma\Big(q^3 \big(a_i(2)+\frac{q_i-q_i^{-1}}{2}(a_i(1))^2\big)E_i E_j-[2]_{q^3}E_i E_j
\big(a_i(2)+\frac{q_i-q_i^{-1}}{2}(a_i(1))^2\big)\\
&&\hspace{4.95cm}+ q^{-3} E_i\big(a_i(2)+\frac{q_i-q_i^{-1}}{2}(a_i(1))^2\big)E_j\Big)K_i\\
=&&\Big(q^3\frac{[4]}{2}\hat{T}_{\omega_i}^{-2}(E_i)E_j+\gamma^{\frac{1}{2}}\frac{[2]}{2}q^3(q-q^{-1})\big( a_i(1)\hat{T}_{\omega_i}^{-1}(E_i)+ \hat{T}_{\omega_i}^{-1}(E_i)a_i(1)\big)E_j\Big)K_i
\end{eqnarray*}

\begin{eqnarray*}
&&\hbox{The third term}\\
=&&\gamma^{\frac{1}{2}}\Big(q^4E_j a_i(1)\hat{T}_{\omega_i}^{-1}(E_i)+q^2E_j \hat{T}_{\omega_i}^{-1}(E_i)a_i(1)-q[2]_{q^3}a_i(1)E_j\hat{T}_{\omega_i}^{-1}(E_i)\\
&&\hspace{0.95cm}-q^2[2]_{q^3}\hat{T}_{\omega_i}^{-1}(E_i)E_j a_i(1)+q a_i(1)\hat{T}_{\omega_i}^{-1}(E_i)E_j+q\hat{T}_{\omega_i}^{-1}(E_i)a_i(1)E_j\Big)K_i\\
=&&\gamma^{\frac{1}{2}}\Big(-q^2E_j a_i(1)\hat{T}_{\omega_i}^{-1}(E_i)-q^5 \hat{T}_{\omega_i}^{-1}(E_i)a_i(1)E_j\\
&&\hspace{1.95cm}+q^2 E_j \hat{T}_{\omega_i}^{-1}(E_i)a_i(1)
+q a_i(1)\hat{T}_{\omega_i}^{-1}(E_i)E_j\Big)K_i
\end{eqnarray*}

\begin{eqnarray*}
&&\hbox{The forth term}\\
=&&\frac{1}{q_i-q_i^{-1}}\Big(q^4E_j\hat{T}_{\omega_i}^{-2}(E_i)(K_i-K_i^{-1})+E_j(K_i-K_i^{-1}) \hat{T}_{\omega_i}^{-2}(E_i)\\
&&\hspace{3.45cm}-[2]_{q^3}(K_i-K_i^{-1})E_j\hat{T}_{\omega_i}^{-2}(E_i)\Big)\\
=&&[4]E_j\hat{T}_{\omega_i}^{-2}(E_i)K_i
\end{eqnarray*}

\begin{eqnarray*}
&&\hbox{The fifth term}\\
=&&\frac{1}{q_i-q_i^{-1}}\Big(-q^4[2]_{q^3}\hat{T}_{\omega_i}^{-2}(E_i)E_j(K_i-K_i^{-1})
+q^4\hat{T}_{\omega_i}^{-2}(E_i)(K_i-K_i^{-1})E_j \\
&&\hspace{3.45cm}+(K_i-K_i^{-1})\hat{T}_{\omega_i}^{-2}(E_i)E_j\Big)\\
=&&-q^3[4] \hat{T}_{\omega_i}^{-2}(E_i)E_j K_i
\end{eqnarray*}

Their total sum is zero, thus we have shown that $[Y, F_k]=0$ and subsequently the Serre relations
hold in this case.

{\bf Case (iii): $A_{ij}=-1$ and $A_{i,\sigma(i)}=0, \sigma(j)\neq
j$.} The required relation
follows from that of the untwisted case verified in Case (i). \\

{\bf Case (iv): $A_{ij}=-1$ and $A_{i,\sigma(i)}=-1$.} This only happens for type $A_{2n}^{(2)}$.
Here $P_{ij}^{\pm}(z_1,z_2)=z_1q^{\pm 1}+z_2$ and $d_{ij}=\frac{1}{2}$,
which is exactly the same as that of  Case (ii) in type
$A_{2n-1}^{(1)}$. Thus the Serre relation is proved by repeating the
argument of Case (ii).

By now we have proved all cases of the Serre relation $(9)$.
\end{proof}

The last Serre relation (10) only exists for type $A_{2n}^{(2)}$.

\begin{prop}\,For $A_{i,\sigma(i)}=0$,
$$Sym_{z_1,z_2,z_3}[(q^{\mp\frac{3r}{4}}z_1-(q^{\frac{r}{4}}-q^{-\frac{r}{4}})z_2+
q^{\pm\frac{3r}{4}}z_3)
x_i^{\pm}(z_1)x_i^{\pm}(z_2)x_i^{\pm}(z_3)]=0.$$
\end{prop}

\begin{proof}\, By the same translation property of the Drinfeld generators, this relation can be replaced by
$$q^{-\frac{3}{2}}x_i^+(1)(x_i^+(0))^2-(q^{\frac{1}{2}}+q^{-\frac{1}{2}})x_i^+(0)x_i^+(1)x_i^+(0)
+q^{\frac{3}{2}}(x_i^+(0))^2 x_i^+(1)=0$$
By definition of $x_i^+(k)$, the above relation is rewritten as:
$$Z:=q^{-\frac{3}{2}}\hat{T}_{\omega_i}^{-1}(E_i)E_i^2
-(q^{\frac{1}{2}}+q^{-\frac{1}{2}})E_i\hat{T}_{\omega_i}^{-1}(E_i)E_i
+q^{\frac{3}{2}}E_i^2\hat{T}_{\omega_i}^{-1}(E_i)=0.$$
Using the same trick of Lemma \ref{lemma1}, we must show that $[\,Z,\, F_k\,]=0$ for all $k\in\hat{I}$.

\begin{eqnarray*}
&&[\,Z,\,F_i\,]\\
&=&\Big(q^{-\frac{3}{2}}\hat{T}_{\omega_i}^{-1}(E_i)[\,E_i,\,F_i\,]E_i
-[2]_i[\,E_i,\,F_i\,]\hat{T}_{\omega_i}^{-1}(E_i)E_i
+q^{-\frac{3}{2}}\hat{T}_{\omega_i}^{-1}(E_i)E_i[\,E_i,\,F_i\,]\Big)\\
&&+\Big(q^{\frac{3}{2}}E_i[\,E_i,\,F_i\,]\hat{T}_{\omega_i}^{-1}(E_i)
-[2]_i E_i\hat{T}_{\omega_i}^{-1}(E_i)[\,E_i,\,F_i\,]
+q^{\frac{3}{2}}[\,E_i,\,F_i\,]E_i\hat{T}_{\omega_i}^{-1}(E_i)\Big)\\
&&+\Big(q^{-\frac{3}{2}}[\,\hat{T}_{\omega_i}^{-1}(E_i),\,F_i\,]E_i^2
-[2]_i E_i[\,\hat{T}_{\omega_i}^{-1}(E_i),\,F_i\,]E_i
+q^{\frac{3}{2}}E_i^2[\,\hat{T}_{\omega_i}^{-1}(E_i),\,F_i\,]\Big)\\
&=&\Big(\frac{[2]_i}{q_i-q_i^{-1}}(q^{-1}-q^2)+q^{\frac{1}{2}}[2]_i(q+q^{-1}+1) \Big)\hat{T}_{\omega_i}^{-1}(E_i)E_i K_i\\
&&+\Big(\frac{[2]_i}{q_i-q_i^{-1}}(q^{3}-1)-q^{\frac{3}{2}}[2]_i(q+q^{-1}+1)\Big)
E_i\hat{T}_{\omega_i}^{-1}(E_i)K_i\\
&=&0.
\end{eqnarray*}
With this last Serre relation we have completed the verification of all Drinfeld relations.
\end{proof}

\section{Isomorphism between the two structures}

\subsection{The inverse homomorphism}\, To complete the proof of Drinfeld realization we need to
establish an isomorphism between the Drinfeld-Jimbo algebra and the Drinfeld new realization.
There have been several attempts to show the isomorphism in the literature,
and all previous proofs only established
a homomorphism from one form of the algebra into the other one. In this section
we will combine our previous approach \cite{J2, ZJ, JZ} together with Beck's idea of braid groups to
finally settle this long-standing problem and prove the isomorphism between the two forms of the quantum
affine algebras in both untwisted and twisted cases.

In order to show there exists an isomorphism between the above two structures, we recall the inverse map of $\Psi$ from Drinfeld realization to twisted quantum affine algebra developed in \cite{ZJ}, where we denoted by $\mathcal{U}_q(\hat{\frak{g}}^{\sigma})$ the Drinfeld realization, which is the associative
algebra over the complex field generated by the elements $x_i^{\pm}(k), a_i(l),\, K_i^{\pm1}$ and $\gamma^{\frac{1}{2}}$, where $i\in \{1,\,2,\, \dots, N\}$, $k\in \mathbb{Z}$, and $l\in \mathbb{Z}/\{0\}$, satisfying the relations $(1)-(10)$.

First of all, we review the notation of quantum Lie brackets from \cite{J2}.

\begin{defi} Let $\mathbb{K}$ be a field and
for $q_i\in \mathbb K^*=\mathbb{K}\backslash \{0\}$ and $i=1,2,\dots, s-1$,
The quantum Lie brackets $$[\,a_1, a_2,\dots,
a_s\,]_{(q_1,\,q_2,\,\dots,\, q_{s-1})}$$ and
$$[\,a_1, a_2, \dots,
a_s\,]_{\la q_1,\,q_2,\,\dots, \,q_{s-1}\ra}$$ are defined
inductively by
\begin{eqnarray*}
\begin{split}
[\,a_1, a_2\,]_{q_1}&=a_1a_2-q_1\,a_2a_1,\\
[\,a_s, a_{s-1}, \dots, a_1\,]_{(q_1,\,q_2,\,\dots,
\,q_{s-1})}&=[\,a_s, \,a_{s-1}, \dots,[\,a_{2},\,
a_1\,]_{q_{1}}\,]_{(q_2,\,\dots,\,q_{s-1})},\\
[\,a_1, a_2, \dots, a_s\,]_{\la q_1,\,q_2,\,\dots,
\,q_{s-1}\ra}&=[\,[\,a_1, a_2\,]_{q_1}, a_3,\,\dots, a_{s}\,]_{\la
q_2,\,\dots,\,q_{s-1}\ra},
\end{split}
\end{eqnarray*}
\end{defi}

To state the inverse homomorphism, we need to fix a particular path to realize the maximum
root of $\mathfrak g_0$.

Let $\theta=\alpha_{i_{h-1}}+\cdots+\alpha_{i_2}+\alpha_{i_1}$ be the maximum root and let
\begin{align}\label{a1}
X_{\theta}=[e_{i_{h-1}}, [e_{i_{h-2}}, \cdots, [e_{i_2}, e_{i_1}]\cdots ]
\end{align}
be the corresponding root vector in the Lie algebra $\mathfrak g_0$, which gives rise to
a sequence from $[1, \cdots, n]$: $i_1, i_2, \cdots, i_{h-1}$. We call such a sequence a
root chain to the maximum root, which is not unique.

From now on we fix a particular path to realize the maximum
root of $\mathfrak g_0$ and the associated sequence $i_1, i_2, \cdots, i_{h-1}$. We define for $2\leqslant k\leqslant h-1$
\begin{align}\label{a2}
(\alpha_{i_1}+\cdots+\alpha_{i_{k-1}}, \alpha_{i_{k}})=\epsilon_k\neq 0.
\end{align}

\begin{theo}\,{\label{t2}  Let $i_1,\,\dots,\, i_{h-1}$ be the sequence of indices
in the particular path realizing the maximum root $\theta$ given in
Eq. (\ref{a1}), then there is an algebra homomorphism $\phi: U_q(\hat{\frak{g}}^{\sigma})\rightarrow\mathcal{U}_q(\hat{\frak{g}}^{\sigma})$ defined by \begin{eqnarray*}
& \phi(E_i)=x_i^+(0),\qquad \phi(F_i)=\frac{1}{p_i}x_i^-(0),\qquad  \phi(K_i)=K_i^+(0),\\
& \phi(E_0)=ax_{\theta}^-(1)\gamma K_{\theta}^{-1},\qquad  \phi(F_0)=\gamma^{-1} K_{\theta}x_{\theta}^+(-1),\qquad
\phi(K_0)=\gamma K_{\theta}^{-1},
\end{eqnarray*}
where $K_{\theta}=K_{i_1}K_{i_2}\cdots K_{i_{h-1}}$ and  $a=(p_{i_2}\cdots p_{i_{h-1}})^{-1}$,
 and $x_{\theta}^-(1)$, $ x_{\theta}^+(-1)$ are defined by quantum Lie bracket as below:
\begin{eqnarray*}
&& x_{\theta}^-(1)=[\,x_{i_{h-1}}^-(0),\, x_{i_{h-1}}^-(0),\, \dots,\, x_{i_2}^-(0),\,x_{i_1}^-(1)\,]_{(q_{i_1}^{\epsilon_1},\,q_{i_2}^{\epsilon_2},\,\dots,\,q_{i_{h-2}}^{\epsilon_{h-2}})}\\
&& x_{\theta}^+(-1)=[\,x_{i_{1}}^+(-1),\, x_{i_{2}}^+(0),\, \dots,\, x_{i_{h-2}}^+(0),\,x_{i_{h-1}}^+(0)\,]_{\langle q_{i_1}^{-\epsilon_1},\,q_{i_2}^{-\epsilon_2},\,\dots,\,q_{i_{h-2}}^{-\epsilon_{h-2}}\rangle}.
\end{eqnarray*}}
\end{theo}

It is not difficult to check that the algebra $\mathcal{U}_q(\hat{\frak{g}}^{\sigma})$
is actually generated by $x_i^{\pm}(0)\, (i=1,\dots, n)$,
$x_{\theta}^-(1)$, $x_{\theta}^+(-1)$, therefore $\phi$ is an epimorphism.

\begin{remark}\, In \cite{ZJ, JZ}, we have checked that $\phi$ is an algebra homomorphism using quantum Lie brackets. In the sequel we show that the
map is in fact the inverse of the action of the braid group.
\end{remark}

\subsection{Isomorphism between two presentations}

Theorem \ref{t1} induces that the homomorphism $\psi$ from Drinfeld realization $\mathcal{U}_q(\hat{\frak{g}}^{\sigma})$ to Drinfeld-Jimbo algebra $U_q(\hat{\frak{g}}^{\sigma})$ is surjective.  We now show its injectivity by checking that the products of two maps are identity.
We start with the following proposition.
\begin{prop}\, With homomorphisms  $\phi$ and $\psi$ defined as above, one has $\psi\phi(E_0)=E_0$.
\end{prop}
\begin{proof}\, Note that
\begin{eqnarray*}
&&\psi\phi(E_0)\\
&=&\psi(ax_{\theta}^-(1)\gamma K_{\theta}^{-1})\\
&=&[\,F_{i_{h-1}},\, F_{i_{h-2}},\, \dots,\, F_{i_2},\,\hat{T}_{\omega_i}(F_{i_1})
\,]_{(q_{i_1}^{\epsilon_1},\,q_{i_2}^{\epsilon_2},\,\dots,\,q_{i_{h-2}}^{\epsilon_{h-2}})}
K_{\delta-\theta}
\end{eqnarray*}
Note that for $a_{ij}=-1$ (see \cite{M}), $$T_i(\hat{T}_{\omega_i}(F_{j}))=-\hat{T}_{\omega_j}(F_{j})F_i+q_iF_i\hat{T}_{\omega_j}(F_{j})
=q_i[F_i,\,\hat{T}_{\omega_j}(F_{j})]_{q_i^{-1}},$$
So the above bracket can be written as $T_{i_{h-1}} T_{i_{h-2}}\cdots  T_{i_{2}}(\hat{T}_{\omega_i}(F_{i_1}))K_{\delta-\theta}$, when $i_1$ is not equal to $i_k$ for all $k=2,\dots, h-1$.

Thus we need to consider the case $a_{i_1\,i_2}=-2$.
\begin{eqnarray*}
&&[F_{i_1},\, [\,F_{i_2},\,\hat{T}_{\omega_{i_1}}(F_{i_1})\,]_{q_{i_1}^{\epsilon_1}}\,]_{q_{i_2}^{\epsilon_2}}
=[\,F_{i_1},\, \hat{T}_{\omega_{i_1}}([\,F_{i_2},\,F_{i_1}\,]_{q_{i_1}^{\epsilon_1}})\,]_{q_{i_2}^{\epsilon_2}}\\
&=&[\,F_{i_1},\, \hat{T}_{\omega_{i_2}}([\,F_{i_1},\,F_{i_2}\,]_{q_{i_1}^{\epsilon_1}})\,]_{q_{i_2}^{\epsilon_2}}
=\hat{T}_{\omega_{i_2}}([\,F_{i_1},\,[\,F_{i_1},\,F_{i_2}\,]_{q_{i_1}^{\epsilon_1}}\,]_{q_{i_2}^{\epsilon_2}})\\
&=&\frac{[2]_{i_1}}{q_{i_1}^{-2\epsilon_1}}T_{i_1}\hat{T}_{\omega_{i_2}}(F_{i_2}),
\end{eqnarray*}
where we have used the $\hat{T}_{\omega_{i_1}}([\,F_{i_2},\,F_{i_1}\,]_{q_{i_1}^{\epsilon_1}})
=\hat{T}_{\omega_{i_2}}([\,F_{i_1},\,F_{i_2}\,]_{q_{i_1}^{\epsilon_1}})$ (see \cite{B}).

Therefore the above q-bracket also can be written as

$$\psi\phi(E_0)=AT_{i'_{k}} T_{i'_{k-1}}\cdots  T_{i'_{2}}(\hat{T}_{\omega_{i'_1}}(F_{i'_1}))K_{\delta-\theta}$$
where $i'_1,\dots ,i'_k$ are in the set $\{i_1,\,\dots i_{h-1}\}$ such that every two elements are different, and $A$ is a polynomial of $q$.

In fact, we have
$$\delta-\alpha_{i_1}-\cdots-\alpha_{i_{h-1}}=\delta-\theta.$$

Recall that we have defined $E_{\delta-\theta}$ as the quantum root vector
$$T_{i'_{k}} T_{i'_{k-1}}\cdots  T_{i'_{2}}(\hat{T}_{\omega_{i'_1}}(F_{i'_1}))K_{\delta-\theta}$$
independently from
the sequence (by Remark \ref{r1}). Thus  $\psi\phi(E_0)=A E_0$ for $\alpha_0=\delta-\theta$, where $A$ is a polynomial of $q$.
Then we can adjust the map $\phi$ such that the action of $\psi\phi$ on $E_0$ is identity.
\end{proof}

 Similarly we can show that $\psi\phi(F_0)=F_0$. Note that by definition $\psi\phi$ and $\psi\phi$ fix all
 $E_i$ and $F_i$ for $i\neq 0$, and also
 the homomorphisms  $\phi$ and $\psi$ are surjective by construction, therefore $\psi\phi=\phi\psi=I$. So we have
 shown that
\begin{coro}\, The homomorphisms $\phi: U_q(\hat{\frak{g}}^{\sigma})\rightarrow\mathcal{U}_q(\hat{\frak{g}}^{\sigma})$ and $\psi: \mathcal{U}_q(\hat{\frak{g}}^{\sigma})\rightarrow U_q(\hat{\frak{g}}^{\sigma})$ are two algebra isomorphisms. In particular $\phi=\psi^{-1}$.
\end{coro}

\bigskip

\centerline{\bf }

\vskip30pt \centerline{\bf COMPETING INTERESTS}

The authors declare that they have no competing interests.

\vskip30pt \centerline{\bf ACKNOWLEDGMENT}
N. Jing would like to thank the partial support of
Simons Foundation grant 198129, NSFC grant (11271138 and 11531004) and NSF grant
(1014554 and 1137837). H. Zhang would
like to thank the support of NSFC grant (11371238 and 11101258).

\bigskip

\bibliographystyle{amsalpha}

\end{document}